\documentclass[12pt]{article}
\usepackage{graphicx}
\usepackage{amsmath}
\usepackage{amsfonts}
\usepackage{amsthm}

\newtheorem{theorem}{Theorem}[section]
\newtheorem{lemma}[theorem]{Lemma}

\newtheorem{exercise}[theorem]{Exercise}

\begin{document}

\title{Jamming as Information: a Geometric Approach}

\author{Tanya Khovanova \\
\textit{Department of Mathematics, MIT}}
\date{August 27, 2008}

\maketitle

\begin{abstract}
In this paper I discuss the kinds of information that can be extracted by our enemy if our jamming is too precise. I show geometric solutions for reconstructing linear routes given certain information about them, such as the shortest distance to a point or the times of entering and exiting a circle. 
\end{abstract}


\section{Introduction}

Suppose we want to fly a plane into enemy territory. To make our enemy less abstract, I would like to call them Zerg (see StarCraft games \cite{StarCraft}). Our enemy has many radar systems to protect their air space. To hide information about the trajectory and location of our plane, we need to jam. At the same time, the fact that we are jamming gives out information that something is going on. As soon as jamming starts, the adrenaline of a certain zergling radar operator goes up and he calls his boss or runs to hide.

How much information do we give out when we jam? It really depends on the jamming strategy. Suppose we jam Zerg radars everyday at 6AM, without trying to fly our planes into their territory. They may grow used to this jamming or may start thinking that there is a malfunction in their system. They may even try to assign this strange periodic noise to new cosmic activities. In any case after a year of daily jamming they will be less alert to our bombers coming in than if we only jam when we are flying into their territory.

Similarly, if we have a rule that we jam a radar system as soon as our plane approaches within 100 kilometers distance of that radar, and the Zerg manage to bribe one of our designers, then, as soon as we start jamming, they would know that our plane is exactly at the 100 km radius from the radar we jam.

How much more information do we give away if we follow some exact jamming rules? 

In this paper I discuss the information that can be deduced in certain simple situations, given our jamming strategy and some assumptions.

The results reported in this paper can not only be used in our war with the Zerg, but can also be applied to different kinds of sensor information. For example, we can be on the other side and try to deduce information when someone jams us. More broadly, we can apply the results of this paper to any kind of sensors that detect movement. For example, we can have sensors around a house that start beeping when a mouse approaches a sensor at a distance of 100 meters. In this case our results can be applied to predict the mouse's route. To make our predictions valid the mouse should obey our assumptions, which might not be extremely realistic, but a reasonably good start for a mathematical model.

\section{Assumptions}

It makes no difference to me as a mathematician if a plane is flying into the radar or a mouse is running towards the sensors. It doesn't even matter if the mouse is flying or the plane is running. The only thing I want is to avoid using words like ``plane'' and ``mouse'' as one can be easily confused with a two-dimensional geometrical object and the other with a computer gadget. From now on I would like to call my flying or crawling object a bomber.

I also do not much care how exactly radars and jamming work. If you are interested in learning more, you can check any radar book \cite{Barton}.

I also would like to reduce my problem to two dimensions. For this I assume that our bombers fly at a constant altitude and that the Zerg know that altitude. Hence, they can effectively project the bomber's trajectory to a horizontal plane. 

Here are all my assumptions:
\begin{enumerate}
\item We have only one bomber in our air force.
\item Our bomber is flying on a horizontal plane.
\item Our bomber flies a constant speed and does not change its direction.
\end{enumerate}

\section{Types of Information}

The Zerg can obtain different types of information from our jamming if we are not accurate and not trying consciously to prevent the Zerg from having extra information.

I would classify the information in this way.

\begin{description}
\item [Time of entering a circle of a fixed radius.] If we start jamming when our secret bomber is exactly at 100 kilometers from the radar and our enemy knows that this is our strategy, then as soon as we start jamming, the Zerg will know that our bomber is entering this circle. Similarly, if motion sensors can turn on when a movement is detected at a given radius, we get this type of information.
\item [Time of exiting a circle of a fixed radius.] If we stop jamming when our bomber exits the threat radius of the radar, the Zerg will know that the bomber exited. At the same time we might know the entering time, but not the exit time. For example, there are some light sensors which turn off after a fixed time. This is why I would prefer to have two separate assumptions for entering and exit times.
\item [Closest distance to a point.] The power with which we need to jam is dependant on many parameters including the size of the bomber and the distance from the bomber to the radar. Suppose we are dumb enough to jam with the power that is exactly the minimum amount we need. Suppose the Zerg know the size of our single bomber and the power formula we use. If so, they might be able to calculate the distance from our bomber to the radar at every moment. As we assume that we are dumb we might not be sophisticated enough to produce a jamming power that changes. Suppose we calculate the exact maximum power we need using the closest distance of the bomber to the radar. Then we jam with the constant power while our bomber is inside the threat circle. In this case the Zerg will be able to calculate the closest distance of the bomber to the radar.
\item [Speed.] If we fly at standard speed and the Zerg have good spies they might know our bomber's speed.
\end{description}

\section{Dimension Considerations}

Suppose we know the time of the bomber entering a given circle. How can we parametrize all possible trajectories of the bomber? We can assign one parameter to a point on a circle and one parameter to the direction of flying. We will need at least two more measurings to calculate the trajectory. To calculate the location we need to know the speed too, which gives us one more measurement that we are seeking.

\section{General Considerations}

\begin{lemma} 
Given that we know entry and exit times, knowledge of the closest distance and of the speed are equivalent.
\end{lemma}
\begin{proof} All chords in a circle that have a fixed distance to the center have the same length; hence knowledge of the closest distance is equivalent to knowledge of the chord's length. And knowledge of the chord's length given the time our bomber covers this distance is equivalent to knowledge of the bomber's speed.
\end{proof}

The next lemma is proved in a similar manner.
\begin{lemma} 
Given that we know the entry time, the closest distance and the speed, we can calculate the exit time. 
\end{lemma}

\begin{exercise} 
Prove the lemma.
\end{exercise}

Now I would like to discuss different cases.

\section{Case 1. Shortest distances to two radars}

\begin{theorem}
Given the shortest distances of the trajectory of the bomber to two different radars, we can calculate up to 4 potential trajectories of the bomber.
\end{theorem}

\begin{proof}
Let us denote by $A$ and $B$ locations of the first and the second radars correspondingly. The situation is symmetrical with respect to the line $AB$. If we build one solution, the $AB$-reflection of it is another solution. Let $C$ and $D$ be the locations of the bomber at the closest distances to the first and the second points respectively. Since all that is known about these points $C$ and $D$ is the distance of each from the respective radar, each point, $C$ or $D$, must lie on a circle centered around the corresponding radar. Since the bomber is flying in a straight line, its trajectory will follow a path tangent to these circles. That means, we need to build mutual tangent lines to two given circles (see Figure \ref{twocirclesTangentLines}).

\begin{figure}[htbp]
  \centering
  \includegraphics[scale=0.5]{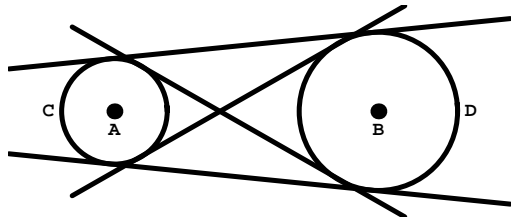}
  \caption{Case 1. Tangent lines for two circles.}\label{twocirclesTangentLines}
\end{figure}

Tangent lines might be internal and external. Let us build the external lines first. Let us translate the trajectory $CD$ with respect to vector $CA$. Point $C$ moves into point $A$ and point $D$ moves into a new point $E$ located on the line $DB$. The line $AE$ touches the circle of radius $BD - CA$ centered at $B$. Hence, to reconstruct $CD$ we need to build a circle of radius $BD - CA$ centered at $B$, then draw a tangent line from point $A$ to this new circle, then move the tangent line perpendicular to itself and away from $B$ at a distance $CA$. (See Figure \ref{twocirclesExternalTangent})

\begin{figure}[htbp]
  \centering
  \includegraphics[scale=0.5]{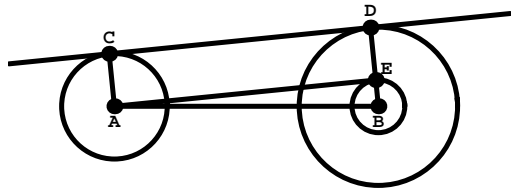}
  \caption{Case 1. Solving for an external tangent line.}\label{twocirclesExternalTangent}
\end{figure}

For internal mutual tangent lines the previous construction becomes the following: build a circle of radius $BD + CA$ centered at $B$, then draw a tangent line from point $A$ to this new circle, then move the tangent line perpendicular to itself and towards $B$ at a distance $CA$. (See Figure \ref{twocirclesInternalTangent})
\end{proof}

\begin{figure}[htbp]
  \centering
  \includegraphics[scale=0.5]{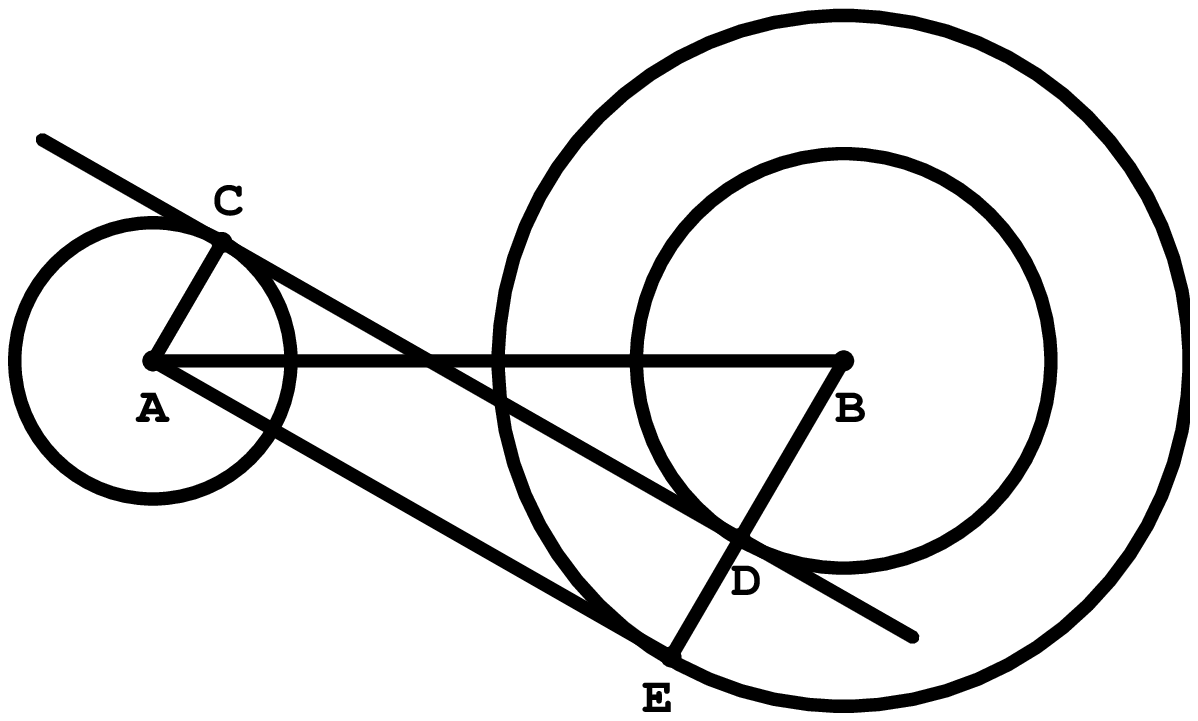}
  \caption{Case 1. Solving for an internal tangent line.}\label{twocirclesInternalTangent}
\end{figure}

\begin{exercise} 
Analyze special cases. When are there no external solutions? When are there no internal solutions? When is there only one external/internal solution? Can we have more than 4 solutions?
\end{exercise}

\begin{exercise} 
Find a different solution to building tangent lines using homothety, also called dilation, homothecy or similarity transformation (see Coxeter and Greitzer \cite{Coxeter} and Yalgom \cite{Yaglom}).
\end{exercise}

\textbf{Discussion.}
Thus, without being able to use a radar to garner even a single detection on a protected mission, and also without knowing any time, just using the knowledge of jamming power, it is possible to narrow the location of the protected bomber down to a finite, and small (four) number of potential routes. If we have a third radar in the area, we might be able to know the exact route. 

Conclusion: 
To completely confuse the Zerg we need to mask the real level of jamming power that we need. It is not enough to increase the power by a factor of 2. Because, if they have the information about 2, the Zerg can divide by 2. To hide the real number we need to use a function that is not invertible --- does not allow anyone to extract the jamming power we need from the power provided. This can be done by randomly adding some extra power or always jamming with the maximum possible power.

\section{Case 2. Shortest distance, entry and exit times to one radars and entry time to another radar}

\begin{theorem}
Given the shortest distance of the trajectory of the bomber to a radar and the time of entering and exiting a threat circle of the radar and the time of entering another threat circle, we can calculate up to 4 potential locations of the bomber at any given time.
\end{theorem}

\begin{proof}
For the first radar the closest distance gives us the length of the chord $L$ on which the vehicle travels inside the first radar's threat circle. The entry and exit times for the first threat circle give us the speed at which the bomber travels. The entrance time to the second threat circlt, given that we already know the speed, gives us the distance $d$ between the entrance point to the first radar's threat circle and the entrance point to the second radar's threat circle. Let us build any chord $AB$ of length $L$ in the first threat circle, then let us continue the line of the chord and find a point $C$ on the line which is in the direction from $A$ to $B$ and at a distance $d$ from $A$. As this picture is rotationally symmetric all the possible potential entrance points for the second radar lie on a circle around the center of the first radar. As we know one point, $C$, we know the circle. As a result, the second entrance point is the intersection of the circle we build with the threat circle of the second radar (see Figure \ref{case2}).

\begin{figure}[htbp]
  \centering
  \includegraphics[scale=0.5]{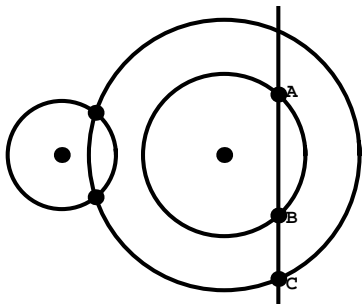}
  \caption{Case 2. Finding entrance points to the second radar's threat circle.}\label{case2}
\end{figure}

From the entrance point to the second circle we can reconstruct the entrance points to the first threat circle as intersections of the first circle with circle of radius $d$ centered at entrance points of the second circle.
\end{proof}

\begin{exercise} 
Show that we can have up to 4 solutions for the location of the bomber. Analyze in which situations we can have less than 4 solutions.
\end{exercise}

\textbf{Discussion.}
If we are providing the Zerg with the exact entrance and exit times and the shortest distance to the radius, they can shoot down our bomber as soon as we exit the first circle or enter the second, whichever comes later. They will need to waste 4 missiles for one bomber as they will shoot in four different directions, but one of these 4 missiles is guaranteed to destroy the bomber.

\section{Case 3. Jamming start time and end time for one radar. One location with time on the mission route}

Suppose the Zerg know nothing about jamming power or speed. Suppose they know the entrance and exit times for a threat radius of a given radar. Suppose for some reason (they intercepted satellite transmission, or a Sith Lord is on their side) the Zerg know the time and location of a point $A$ on the route of the bomber. 

\begin{theorem}
Given the entrance and exit times for a given circle and a location with time on the route of the bomber, we can calculate the complete route.
\end{theorem}

\begin{proof}
Let us assume that point $A$ is outside of the threat circle -- or, equivalently, the time at the given location $A$ is outside the entrance-exit time interval. Suppose the time between our bomber being at point $A$ and entrance time is $x$ and the time between $A$ and exit time is $y$. Then we do a homothety of a threat circle with the center $A$ and the ratio $y/x$. You can read about homothety in more detail on the wiki \cite{Wiki}, Mathworld \cite{MW} and Coxeter \cite{Coxeter}. Essentially, our homothety operation is the following: We take the center $O$ of our threat circle and build a point $O_1$  on the line $AO$ such that $O_1A/OA = y/x$. Then we build a circle with the center $O_1$ and the radius $yr/x$, where $r$ is the radius of the threat circle. The point of intersection of the threat circle with the new circle we built is the exit point. (See Figure \ref{case3})

\begin{figure}[htbp]
  \centering
  \includegraphics[scale=0.5]{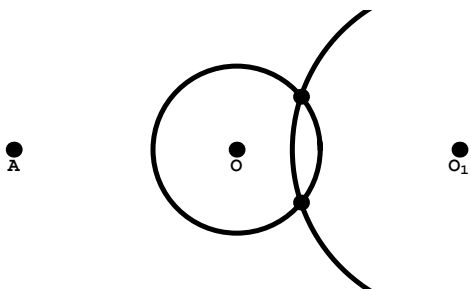}
  \caption{Case 3. Finding exit points.}\label{case3}
\end{figure}

Let us denote the exit point we built by $Y$ and the second intersection point of the line $AY$ with the threat circle as $X$. Then $X$ is the entrance point to the threat circle. And the complete route is the line $AXY$. (See Figure \ref{case3_2})

\begin{figure}[htbp]
  \centering
  \includegraphics[scale=0.5]{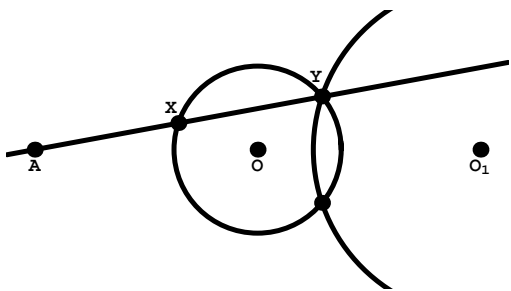}
  \caption{Case 3. Recovering the entrance point from the exit point.}\label{case3_2}
\end{figure}

\end{proof}

\begin{exercise} 
Show that for the points $Y$ and $X$ built in the proof above, $AY/AX = y/x$.
\end{exercise}

\begin{exercise} 
Finish the proof when $A$ is inside the threat circle.
\end{exercise}

\begin{exercise} 
How many solutions can we get? Provide the analysis.
\end{exercise}

\section{Case 4. Jamming start time and end time for two radars}

Suppose we know nothing about shortest distances to radars. 

\begin{theorem}
Given entrance and exit times for two circles we can calculate a potential trajectory.
\end{theorem}

\begin{proof}
Suppose we found intersection points $A$, $B$, $C$ and $D$ of the trajectory with our two threat circles (See Figure \ref{case4}). As we assume that the speed is constant we can derive the ratios $BC/AB = x$ and $CD/AB = y$ from the known times. 

\begin{figure}[htbp]
  \centering
  \includegraphics[scale=0.5]{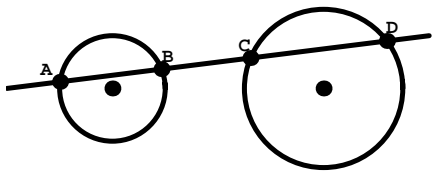}
  \caption{Case 4. Analysis.}\label{case4}
\end{figure}

This means we need to find the line such that its intersection by the two given circles divides the line into segments of given ratios. Let us solve this problem in coordinates. There are many ways to carry out the calculation.

First approach: We can parametrize any line with two parameters. Then we calculate the intersection points $A$, $B$, $C$ and $D$ as the function of these two parameters. The given ratios $AB/BC$ and $BC/CD$ give us two equations to solve to find the values of our two parameters.

Second approach: Let us parametrize points on the first circle with one real parameter. If we know point $A$ on the first circle we can find point $C$ on the second circle as the intersection of the second circle with the homothetic expansion of the first circle with center $A$ and ratio $x+1$. (See Figure \ref{case4_2}) Similarly, we can find point $D$ on the second circle as the intersection of the second circle with the homothety of the first circle with center $A$ and ratio $x+y+1$. Now our equation for our parameter is using the fact that these two different intersection points $C$ and $D$ should lie on the same line with $A$.

\begin{figure}[htbp]
  \centering
  \includegraphics[scale=0.5]{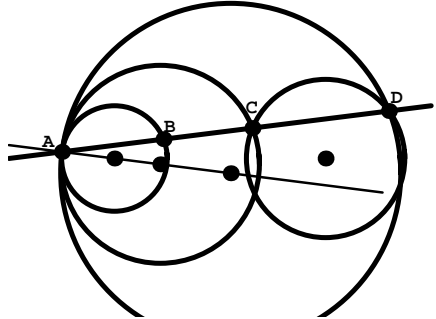}
  \caption{Case 4. The second approach.}\label{case4_2}
\end{figure}

Third approach: Let us denote the center of $AB$ as $P$ and the distance from the center of the first circle to $P$ as $p$. Then we can calculate the length of this chord $AB$ as $\sqrt{R_1^2 - p^2}$, where $R_1$ is the first radius. From the ratios we can now calculate the chord $CD$ and, correspondingly, the distance $q$ to it from the center of the second circle. Using the ratios we can also calculate the distance $s$ between the centers of the chords $AB$ and $CD$. Now we can express the distance between the centers of the circles in terms of $p$, $q$ and $s$. This gives us an equation on our initial parameter $p$. (See Figure \ref{case4_3})
\end{proof}

\begin{figure}[htbp]
  \centering
  \includegraphics[scale=0.5]{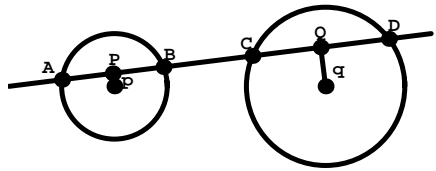}
  \caption{Case 4. The third approach.}\label{case4_3}
\end{figure}

\begin{exercise}
Which method of calculation do you like better? Can you suggest other methods to calculate the trajectory?
\end{exercise}

\section{Speed}

The Zerg can know our speed. It can happen if we always fly with the maximum speed or the same speed. The Zerg do not even need very good spies. They just need to see us flying once. 

If the speed is known and the entrance time is given, the shortest distance to the radar center knowledge is equivalent to exit time knowledge. For example, if the Zerg are using our jamming power to calculate the shortest distance to the radar they might know our exit time before we actually exit.

\begin{exercise}
Suppose the Zerg know the speed of our bomber. If their only other knowledge is the entrance times to the threat radii, for how many radars do they need to know entrance times to reduce the potential bomber trajectory to a finite number of lines?
\end{exercise}

\section{Relaxing Assumptions}

Let us try to relax assumptions and see what happens. Suppose that we increase our air fleet significantly, and now have two bombers. Suppose our jamming tactic is to calculate the jamming for each bomber separately and then use the super-set of the jamming power to deliver the actual power. For example, our first bomber enters the threat radius of the Zerg radar at time $t_1$, then at time $t_2$ our second bomber enters, then they exit the circle at times $t_3$ and $t_4$ correspondingly. 

If $t_3$ is later than $t_4$, then we jam starting at $t_1$ and ending at $t_3$. In this case our jamming correspond to the first bomber. Even if the Zerg do not have any information about the existence of the second bomber, they can use jamming information to find and kill the first bomber. Then the second bomber becomes the only bomber and the problem is reduced to the previous case.

If $t_3$ is earlier than $t_4$, then we jam starting at $t_1$ and ending at $t_4$. If the Zerg have outdated information and think we have only one bomber, their calculation becomes completely invalid, they get confused and we have a chance to win. If they know that we have two bombers, they might start considering different cases and combinations. It will take them longer to get our routes, but they have a chance.

Suppose now we allow our bomber to change speed without changing its direction. Some of the cases we calculated do not consider speed as in case 1. If we know the trajectory, then the constant speed identifies the location as soon as one point is given. If the speed varies and one point is given, then our potential location becomes a segment on the trajectory. The bigger the variation and the farther away from a given point in time the bigger the potential segment. For other cases, the potential trajectory can become a range of trajectories.

Suppose we allow the bomber to change direction. In the next section I would like to consider a bonus case where our bomber changes its direction, but not its speed, exactly once.

\section{Bonus Case}

Let us assume that our bomber starts at a given point in a known direction with a given speed. Then it changes the direction of flying without changing the absolute speed. Suppose the bomber enters the threat circle when it is on its second leg. Suppose the Zerg know the starting point and time and the starting direction. Suppose they also know the total distance from the starting point to the entrance point and the distance inside the threat circle.

\begin{theorem}
In this case the Zerg can recover the location of the bomber.
\end{theorem}

\begin{proof}
Here is what the Zerg know (see Figure \ref{bonusCase}): They know the location and time of point $A$ and the threat circle. They also know the direction from $A$ to $B$ and the distance $CD$. Also they know the sum $AB + BC = d$. The task is to actually find the points $B$, $C$ and $D$.

\begin{figure}[htbp]
  \centering
  \includegraphics[scale=0.5]{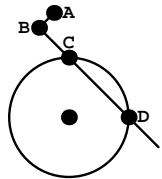}
  \caption{Bonus Case. The setting.}\label{bonusCase}
\end{figure}

Now let us solve another problem. Suppose the bomber starts from point $A$ in the direction to $B$ at a given time with a given speed. Where should the threat circle of the same radius as our circle be located so that the entrance point is at a distance $d$ from $A$ and the trajectory inside the circle is equal to $CD$? (See Figure \ref{bonusCase2}).

\begin{figure}[htbp]
  \centering
  \includegraphics[scale=0.5]{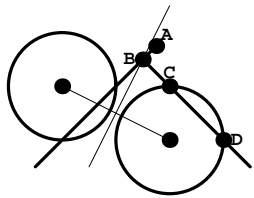}
  \caption{Bonus Case. Solution.}\label{bonusCase2}
\end{figure}

\begin{exercise}
Construct such a circle.
\end{exercise}

Now, it is easy to prove that the solution to our initial problem is the rotation of the trajectory past point $B$ around point $B$. This means the point $B$ belongs to the perpendicular bisector of the segment connecting the radar and the center of the just constructed circle.
\end{proof}

\begin{exercise}
How many different answers can there be to this problem?
\end{exercise}

\section{Acknowledgements}

I am grateful to Jane Sherwin and Sue Katz for editing this paper. I am also thankful to the Art of Problem Solving for inviting me to give an online lecture on geometric transformations at their WOOT program. This lecture served as the encouragement for finishing this paper. I also would like to thank Alexey Radul for explaining to me the ways of the Zerg.


\begin{thebibliography}{9}

\bibitem{Barton} David K.~Barton, "Radar System Analysis and Modeling." Artech House, Boston. 2005.

\bibitem{Coxeter} H.~S.~M.~Coxeter, S.~L.~Greitzer, Geometry Revisited, 1967, p. 94.

\bibitem{StarCraft} Blizzard Entertainment --- StarCraft. \emph{http://www.blizzard.com/\-us/\-star\-craft/}

\bibitem{MW} Eric W.~Weisstein, "Circle-Circle Tangents." From MathWorld -- A Wolfram Web Resource. \emph{http://mathworld.wolfram.com/\-Circle-Circle\-Tangents.html} 

\bibitem{Wiki} Wikipedia, "Homothetic transformation." \emph{http://en.wikipedia.org/\-wiki/\-Homothetic\_transformation}

\bibitem{Yaglom} I.~M.~Yaglom, Geometric Transformations II, 1968, pp. 9-36.



\end{thebibliography}
\end{document}